\documentclass[11pt]{amsart}

\usepackage{amsmath,amssymb,amsthm}
\usepackage{tikz}
\long\def\eatit#1{}
\newtheorem{Thm}{Theorem}[section]

\newtheorem{Prop}[Thm]{Proposition}
\newtheorem{Lem}[Thm]{Lemma}
\theoremstyle{definition}
\newtheorem{Def}[Thm]{Definition}
\newtheorem{Ex}[Thm]{Example}
\newtheorem{Rmk}[Thm]{Remark}

\newcommand{\RR}{{\mathbb{R}}}

\newcommand{\NN}{{\mathbb{N}}}

\newcommand{\TT}{{\mathcal{T}}}

\newcommand{\PPP}{{\mathbb{P}}}

\begin{document}
\title{Waldschmidt constants for Stanley-Reisner ideals of a class of polytopes}

\author{Cristiano Bocci and Barbara Franci}

\address{Cristiano Bocci\\
Department of Information Engineering and Mathematics, University of Siena\\
Via Roma 56, 53100  Siena, Italy}
\email{cristiano.bocci@unisi.it}

\address{Barbara Franci\\
Department of Mathematics, Politecnico di Torino \\
Corso Duca degli Abruzzi 24,  10129 Torino, Italy}
\email{barbara.franci@polito.it}

\begin{abstract}
We study the symbolic powers of  the Stanley-Reisner ideal $I_{B_n}$ of a bipyramid $B_n$ over a $n-$gon $Q_n$. 
Using a combinatorial approach, based on analysis of subtrees in $Q_n$ we compute the Waldschmidt constant of $I_{B_n}$.
\end{abstract}


\maketitle

\section{Introduction}
Comparing the behavior of symbolic and regular powers of a homogeneous ideal has become an important key to understanding many problems in commutative algebra and algebraic geometry. 
To be more precise, let $R=k[x_0, \dots, x_n]$, where $k$ is a field and $0\not= I\subset R$ be a homogeneous ideal. We define the $m-$th symbolic power of $I$ to be
\[
I^{(m)}=R\cap \bigcap_{P\in Ass(I)}I^mR_P.
\]
The containment problem concerns with the study of pairs $(m,r)$ such that $I^{(m)}\subset I^r$.

Among the known results is the celebrated containment of Ein-Lazarsfeld-Smith \cite{ELS} and Hochster-Huneke \cite{HH} stating that $I^{(er)}\subset I^r$, where $e$ is the codimension of the ideal $I$.  Much effort has been put towards tightening this containment (\cite{BCH,BH,HHu}) and many conjectures arise ( \cite{BCH,DST,HHu, HS}).
The approach introduced in \cite{BH} permits to use the informations given by the invariants of $I$ to study the containment problem for $I$. Among these invariants there is the so called Waldschmidt constant $\gamma(I)$:
$$\gamma(I)=\lim_{m \to \infty}\alpha(I^{(m)})/m.$$   
where $\alpha(I^{m)})=min\{t:I^{(m)}_t\neq 0\}$.

In the case of an ideal of points, $\gamma(I)$ is strictly related to the multi-point seshadri constant.
However,  this invariant is hard to compute and only few cases are known.

Recently, in \cite{Cooper}, the authors study the containment problem for the class of square-free monomial ideals and prove that all conjectures in \cite{HHu} hold for this class.  Such class of ideals is strictly connected with graphs and simplicial complex: each graph $G$ and  each simplicial complex $\Delta$ have associated square-free monomial ideals called respectively the edge ideal $I_G$ of $G$ and Stanley-Reisner ideal $I_\Delta$ of $\Delta$. Viceversa, to each square-free monomial ideal, we can associate a graph and a simplicial complex.

In literature there are many papers concerning the relationships between the combinatorial properties of the graph or of the simplicial complex and the algebraic properties of their ideals (see \cite{villa} for a survey). It is know, for example, that $I_G^{(m)}=I_G^m$ if and only $G$ is a bipartite graph. 
Moreover the regularity of $R/I_G$ is strictly related to the induced matching of $G$.

Moving in this direction, it is of a sort of interest to investigate if the Waldschmidt constant of a monomial ideal can be expressed in term of combinatorial data of the graph or of the simplicial complex.
This short paper wants to be a starting points in this research. Here we focus on the Stanley-Reisner ideal of a particular class of simplicial complexes: a bipyramid $B_n$ over a $n-$gon.
This class represents a first non trivial case where we can use a pure combinatorial approach to the study of $\alpha(I_{B_n}^{(m)})$. In particular, using a set of subtrees contained in the base $n-$gon, we are able to establish a behavior of the degree of a minimal generator for some symbolic powers of $I_{B_n}$, depending on the parity of $n$:
\begin{itemize} 
\item[i)] $\alpha(B_{2k}^{(s(k-1))})=sk$
\item[ii)] $\alpha(B_{2k-1}^{(s(k-3))})=s(2k-1)$.
\end{itemize}

Using these equalities we  prove the following result.
\begin{Thm}\label{main}
$$\gamma(I_{B_n})=\frac{n}{n-2} \hspace{.2cm} \forall n.$$
\end{Thm}
In particular, even though the computation of the $\alpha$ reveals different behavior according if the number $n$ of vertices in the base is even or odd, the Waldschmidt constant depends only on $n$.

\section{Preliminaries}

\subsection{Simplicial complexes and Stanely-Reisner ideals} We assume that the basic facts about polytopes and simplicial complexes are known and we suggest \cite{Zie} and \cite{Miller} as references.

We start recalling the definition of bipyramid.
\begin{Def}
The bipyramid over a polytope $P$, $bipyr(P)$ is the convex hull of $P$ and any line segment that pierces the interior of $P$ at precisely one point. 
\end{Def}

\begin{Ex}\rm
Assume that the origin in $\RR^d$ is in the interior of $P$ and embed $P$ in $\RR^{d+1}$. Let $T$ be the segment $[-e_{d+1}, e_{d+1}]\subset \RR^{d+1}$, where $e_{d+1}$ represent the $(d+1)$-th coordinate point. Then $bipyr(P)$ is the convex hull conv$(P,T)$.
\end{Ex}

\begin{Def}\label{nostrobn}
Let $Q_n$ be a $n-$gon in $\RR^2$, with vertices $\{1, \dots, n\}$ containing the origin and embedded in $\RR^3$. We denote by $B_n$ the bipyramid over $Q_n$, that is the convex hull 
\[
B_n=\mbox{conv}\left(Q_n, (0,0,1), (0,0,-1)\right).
\]
The two extra vertices of $B_n$, in the point $(0,0,1)$  and $(0,0,-1)$ will be labeled by $0$ and $n+1$. We will refer to  $\{0\}$ and $\{n+1\}$ as, respectively, the upper and lower vertex of  $B_n$.
\end{Def}

\begin{Ex}\rm
The following figures show respectively $B_5$ and $B_8$.
\begin{center}
\begin{tabular}{ccc}
\includegraphics[width=5.4cm]{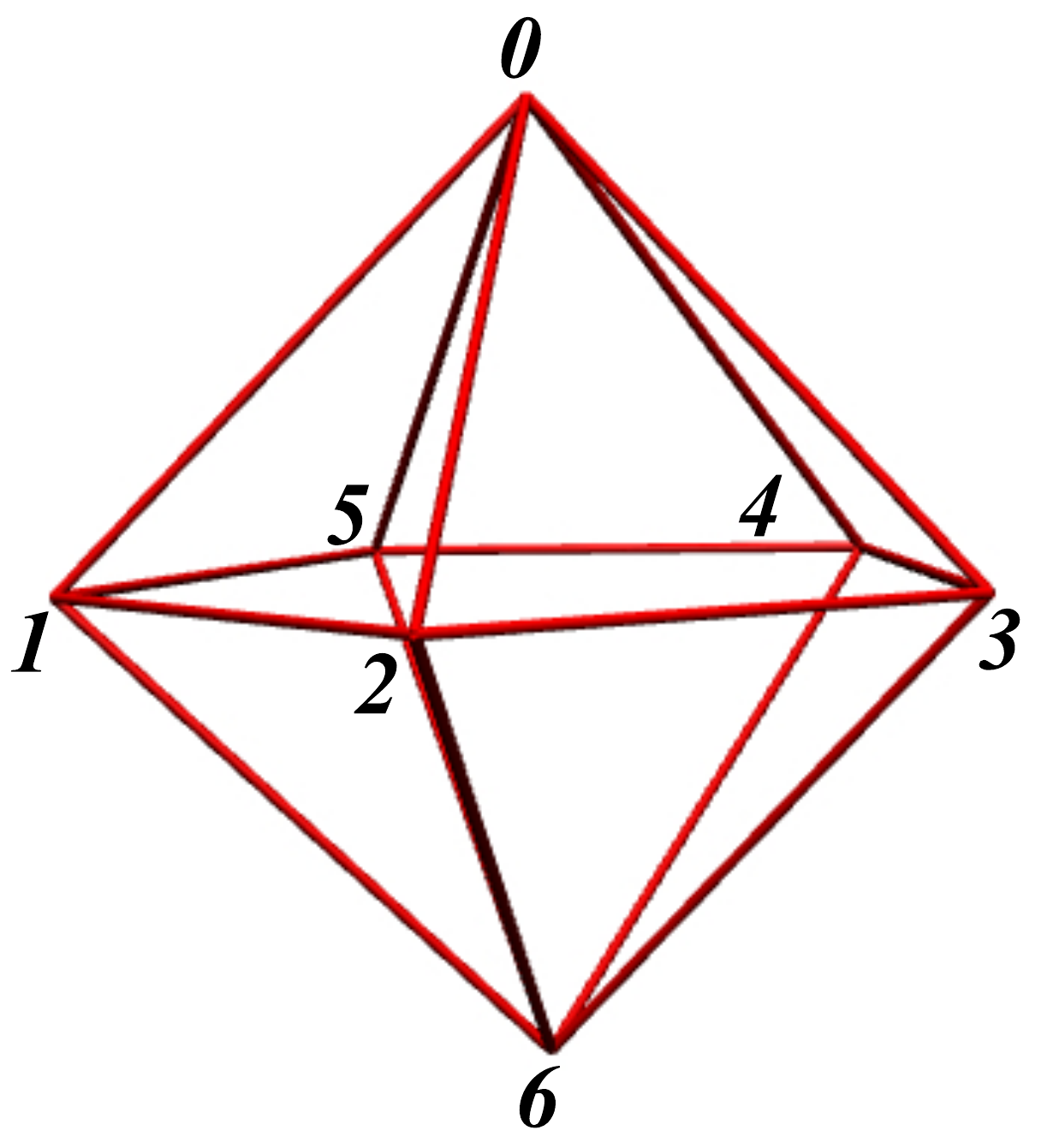} & \hskip1cm & \includegraphics[width=6.2cm]{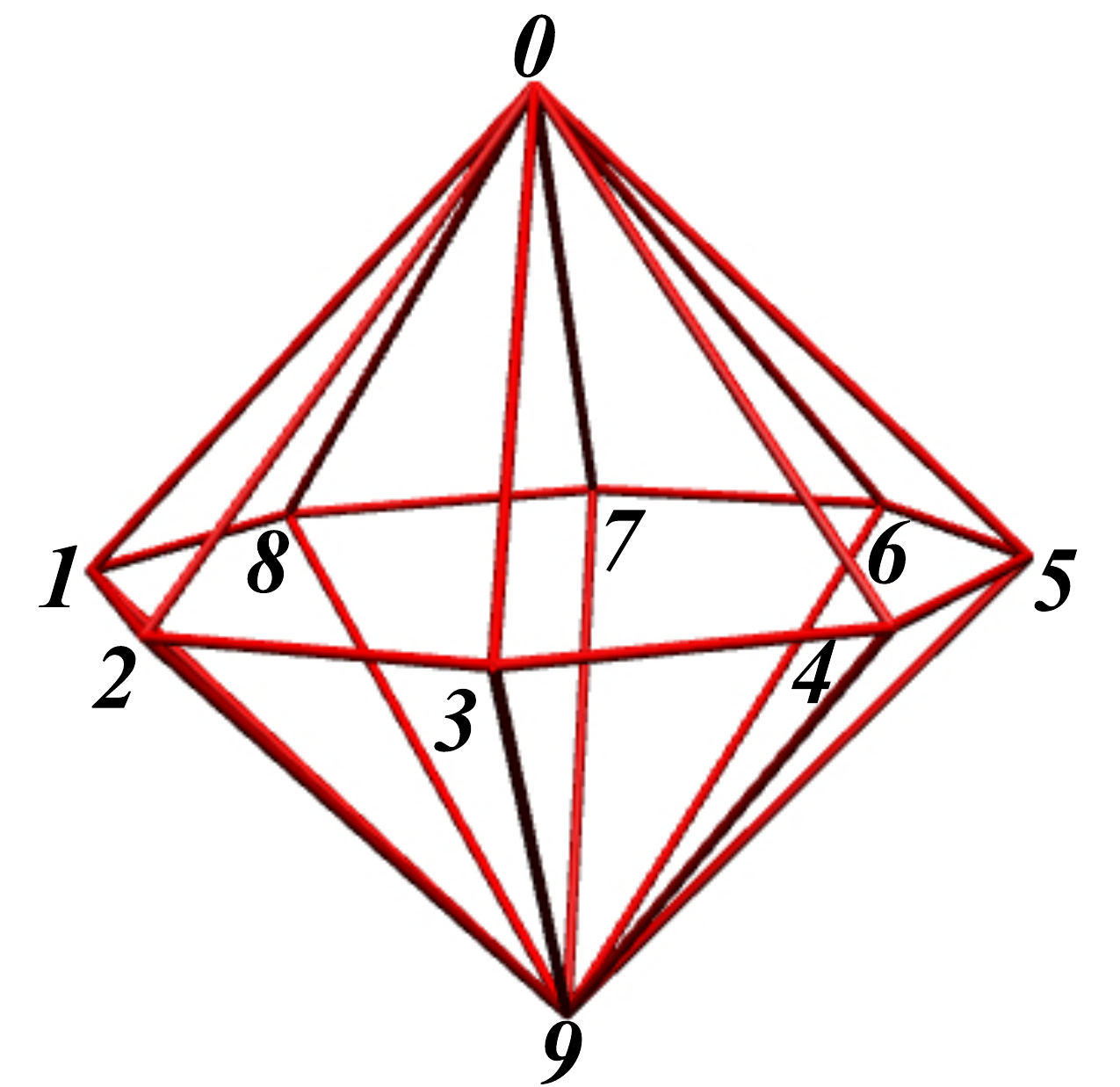}
\end{tabular}
\end{center}
\end{Ex}

Due to their particular construction, the polytopes $B_n$ can be seen also as simplicial complexes and then we can construct their associated Stanley-Reisner ideal.

We recall that for a simplicial complex $\Delta$  with vertices $\{1, \dots, N\}$, we can identify each subset $\sigma \subseteq \{1, \dots, N\}$ with its squarefree vector in $\{0,1\}^N$, which has entry  $1$ in the $i-$th spot when $i\in \sigma$, and $0$ in all other entries. This convention allows us to write $x^\sigma=\Pi_{i\in \sigma}x_i$.

 \begin{Def}\label{DefSR}
The Stanley-Reisner ideal of the simplicial complex $\Delta$ is the squarefree monomial ideal
\[
I_\Delta=<x^\tau\, ; \, \tau\notin \Delta>\subset k[x_1, \dots, x_N],
\]
generated by monomials corresponding to non-faces $\tau$ of $\Delta$.
\end{Def}

There are two ways to present a squarefree monomial ideal: either by its generators or as an intersection of monomial prime ideals. These are generated by subsets of $\{x_1, \dots, x_N\}$. For notation, we write
\[
m^\tau=<x_i\, :\, i\in \tau>
\]
for the monomial prime ideal corresponding to $\tau$. Frequently, $\tau$ will be the complement $\overline{\sigma}=\{1, \dots, N\}\setminus \sigma$ of some simplex $\sigma$.

\begin{Thm}[\cite{Miller}, Theorem 1.7]\label{1.7} The correspondence $\Delta \rightsquigarrow I_\Delta$ constitutes a bijection from simplicial complexes on vertices $\{1, \dots, N\}$ to squarefree monomial ideals inside $A=k[x_1, \dots, x_N]$. Furthermore,
\[
I_\Delta=\bigcap_{\sigma\in \Delta}m^{\overline{\sigma}}.
\]
\end{Thm}

\subsection{Symbolic powers of monomial ideals}

The study of the Waldschmidt constant of an ideal $I$ concerns with the analysis of generators of minimal degree in the  symbolic powers $I^{(m)}$ of $I$. The general definition of the $m-$th symbolic power of $I$ is based on the intersection of the $m-$th ordinary powers of the minimal components in its primary decomposition. In the case of squarefree monomial ideals the situation is simpler.

\begin{Prop}
Suppose $I=Q_1 \cap \cdots \cap Q_k$ is a primary decomposition  of a monomial ideal and $P$ is an associated prime of $I$. Let $Q_{\subseteq P}$ be the intersection of all $Q_i$ with $\sqrt{Q_i} \subseteq P$. For all $m$, $Q^m_{\subseteq P}=R \cap I^mR_P$.
\end{Prop}

\begin{proof}
See \cite{Cooper}, Proposition 3.6.
\end{proof}
A consequence of the previous proposition with  $m=1$ is that $Q_{\subseteq P}=R \cap IR_P$, and thus $Q_{\subseteq P}$ does not depend on a choice of primary decomposition. Henceforth, we will use $Q_{\subseteq P}=R \cap IR_P$ as the definition of $Q_{\subseteq P}$ so that we may avoid choosing a primary decomposition.

\begin{Thm} \label{decomp}
The $m-$th symbolic power of a monomial ideal is
$$I^{(m)} = \bigcap_{P \in maxAss(I)} Q^m_{\subseteq P}$$
where $Q_{\subseteq P} = R \cap IR_P$
\end{Thm}

\begin{proof}
See \cite{Cooper}, Theorem 3.7.
\end{proof}

\begin{Def} 
Given a homogeneous ideal $0\neq I\subseteq R=k[\PPP^N]$, we denote by $\alpha(I)$ the minimum degree of a generator of $I$, i.e. $\alpha(I)=min\{t:I_t\neq 0\}$.
\end{Def}

\begin{Def} \label{defW}
Given a homogeneous ideal $0\neq I \not\subseteq k[\PPP^N]$, the  Waldschmidt constant of $I$ is the limit
$$\gamma(I)=\lim_{m \to \infty}\alpha(I^{(m)})/m.$$   
\end{Def}

A proof of the existence of the limit can be found in \cite{HR2}.

\section{Computations on the ideals $I_{B_n}$}

The final aim of this section is to prove Theorem \ref{main}. As soon as we want to compute the Waldschmidt constant, we need to understand the behavior of the $m-$th symbolic powers of  $I_{B_n}$.

\subsection{Description of $I_{B_n}$} As a first step we describe the generators of $I_{B_n}$.
\begin{Prop}\label{generators}
Given $B_n$, the generators of $I_{B_n}$ are the squarefree monomials  $x_ix_j$, where $i$ and $j$ represent two non-adjacent vertices of $B_n$
\end{Prop}

\begin{proof}
It follows directly from Definition \ref{DefSR}.
\end{proof}

Let $\TT_{n-2,Q_n}$ be the set of sub-trees in the $n-$agon $Q_n$, with $n-2$ vertices.
Given $S\in \TT_{n-2,Q_n}$, we define the following ideals
\[
I_{[0,S]}=\langle x_0\rangle+\langle x_i \, : \, i \in V(S)\rangle, \qquad I_{[n+1,S]}=\langle x_{n+1}\rangle+\langle x_i \, : \, i \in V(S)\rangle.
\]

\begin{Prop}\label{decompo}
The primary decomposition of $I_{B_n}$  is
\begin{equation}\label{pridec1}
I_{B_n}=\bigcap_{S\in \TT_{n-2,Q_n}}\left(I_{[0,S]}\cap I_{[n+1,S]} \right).
\end{equation}
In particular, the $m-$th symbolic powers of i $I_{B_n}$ is
\begin{equation}\label{pridec2}
I_{B_n}^{(m)}=\bigcap_{S\in \TT_{n-2,Q_n}}\left(I_{[0,S]}^m\cap I_{[n+1,S]}^m \right).
\end{equation}

\end{Prop}

\begin{proof}
The decomposition (\ref{pridec1}) follows directly by Theorem \ref{1.7}. The decomposition (\ref{pridec2}) of $I_{B_n}^{(m)}$ follows from Theorem \ref{decomp} since $\sqrt{I_{[0,S]}}=I_{[0,S]}$, $\sqrt{I_{[n+1,S]}}=I_{[n+1,S]}$ for all $S\in \TT_{n-2,Q_n}$ and $maxAss(I)=\{I_{[0,S]},I_{[n+1,S]}\, : \,S\in \TT_{n-2,Q_n} \}$.
\end{proof}

Thus, by the previous proposition, there is a 2:1 correspondence between primary components of the ideal  $I_{B_n}$ and the  $(n-2)-$trees on the base $Q_n$ of $B_n$.

\begin{Ex}\rm
Consider the tree $S$, with vertices $\{3,4,5,6,7,8\}$, in $\TT_{6,Q_8}$, where $Q_8$ is the base of $B_{8}$.
Then the primary components associated to $S$ are
\[
I_{[0,S]}=\langle x_0, x_3,x_4,x_5,x_6,x_7,x_8\rangle \qquad 
I_{[9,S]}=\langle x_3,x_4,x_5,x_6,x_7,x_8,x_9\rangle.
\]
\end{Ex}

\begin{Def}\label{pesi}
Let  $f$ be the monomial $x_0^{a_0}x_1^{a_1}\cdots x_{n+1}^{a_{n+1}}$ and consider a tree $S$ in $\TT_{n-2,Q_n}$. 
We define the weights  $w_{0,S}(f)$, $w_{n+1,S}(f)$ of $f$ with respect to  $S$ and respectively to $x_0$ and $x_{n+1}$ as
\[
w_{0,S}(f)=\sum_{i\in S}a_i+a_0
\]
and
\[
w_{n+1,S}(f)=\sum_{i\in S}a_i+a_{n+1}
\]
\end{Def}

\begin{Ex}\rm
Consider $B_9$ and let $f$ be the monomial $x_0^3x_3^2x_5x_7^2x_{10}$. 
Let  $S$ and $S'$ the subtrees in $\TT_{7,Q_9}$ of vertices respectively $\{1,2,3,4,5,6,7\}$ and $\{1,4,5,6,7,8,9 \}$. Then
\[
\begin{array}{ll}
 w_{0,S}(f)=8& w_{10,S}(f)=6 \\
 w_{0,S'}(f)=6& w_{10,S'}(f)=4. 
\end{array}
\]
\end{Ex}
Given a monomial $f$, we can use the weights previously defined to determine if $f\in I_{B_n}^{(m)}$.
\begin{Prop}\label{carat2}
A monomial $f\in k[x_0, \dots, x_{n+1}]$ is in $I_{B_n}^{(m)}$ if and only if
\[
\min\{w_{0,S}(f), w_{n+1,S}(f)\}\geq m, \quad \forall S\in \TT_{n-2,Q_n}.
\]
\end{Prop}

\begin{proof}

By Theorem \ref{decompo}, $f\in I^{(m)}$ if and only if $f\in I_{[0,S]}^m\cap I_{[n+1,S]}^m$,  $\forall S\in \TT_{n-2,Q_n}$. 
This happens if and only if $f$ is divisible for a monomial in $I_{[0,S]}^m\cap I_{[n+1,S]}^m$,  $\forall S\in \TT_{n-2,Q_n}$, which is equivalent to ask that $\min\{w_{0,S}(f), w_{n+1,S}(f)\}\geq m$, for all $S\in \TT_{n-2,Q_n}$.
\end{proof}

\begin{Rmk}\label{rmksomma}\rm Given two monomials $f,g\in k[x_0,\dots, x_{n+1}]$ with
\[
\min\{w_{0,S}(f), w_{n+1,S}(f)\}\geq m_1, \quad \forall S\in \TT_{n-2,Q_n}.
\]
and
\[
\min\{w_{0,S}(g), w_{n+1,S}(g)\}\geq m_2, \quad \forall S\in \TT_{n-2,Q_n}.
\]
it is easy to verify that
\[
\min\{w_{0,S}(gg), w_{n+1,S}(fg)\}\geq m_1+m_2, \quad \forall S\in \TT_{n-2,Q_n}.
\]
\end{Rmk}

We give now two lemmas which describes some behavior of the exponents of the variables in a monomial $f$, with respect to the condition  $f\in I^{(m)}_{B_n}$.

\begin{Lem}\label{x0xn+1}
Let $f=x_0^{a_0}x_1^{a_1}\cdots x_{n+1}^{a_{n+1}}$, if  $f\in I^{(m)}_{B_n}$, then $g=x_0^tx_1^{a_1}\cdots x_{n+1}^t\in  I^{(m)}_{B_n}$, where $t=\min\{a_0,a_{n+1}\}$.
\end{Lem}

\begin{proof}
Given $f=x_0^{a_0}x_1^{a_1}\cdots x_{n+1}^{a_{n+1}}$, we can suppose that $a_0=t=\min\{a_0,a_{n+1}\}$. Then
\[
w_{0,S}(f)=\sum_{x_i\in S}a_i+t\leq \sum_{x_i\in S}a_i+a_{n+1}= w_{n+1,S}(f)
\]
Since  $f\in I^{(m)}_{B_n}$, by Proposition \ref{carat2}, one has
\[
w_{0,S}(f)=\min\{w_{0,S}(f), w_{n+1,S}(f)\}\geq m, \quad \forall S\in \TT_{n-2,Q_n}
\]
Consider $g=x_0^tx_1^{a_1}\cdots x_{n+1}^t$, then
\[
w_{0,S}(g)=w_{n+1,S}(g)=w_{0,S}(f), \quad \forall S\in \TT_{n-2,Q_n}
\]
hence, again by Proposition \ref{carat2},  $g\in I^{(m)}_n$.
\end{proof}

\begin{Lem}\label{x1xn}
Let $f=x_0^{a_0}x_1^{a_1}\cdots x_{n+1}^{a_{n+1}}$ and define $g=x_0^{b_0}x_1^{b_1}\cdots x_{n+1}^{b_{n+1}}$ where $b_i=\max\{a_i-1,0\}$ for $i=1,\dots,n$ If  $f\in I^{(m)}_{B_n}$, then $g=x \in  I^{(m-(n-2))}_{B_n}$. Moreover $\deg(g)\geq \deg(f)-n$ and equality holds if $a_i\not=0$ for $i=1,\dots, n$.
\end{Lem}

\begin{proof}
Since every tree in $\TT_{n-2,Q_n}$ involves  $n-2$ variables one has
\[
w_S(g) \geq w_S(f)-n-2 \geq m-n-2
\]
and the first statement follows.
For the second statement, let $t=\sharp\{a_i=0, i=1,\dots, n\}$, then $\deg(g)\geq \deg(f)-n+t$.
\end{proof}

\begin{Rmk}\label{impo}\rm Thank to Lemma \ref{x0xn+1}, we can restrict to study the monomials $f=x_0^{a_0}x_1^{a_1}\cdots x_{n+1}^{a_{n+1}}$ where $a_0=a_{n+1}$. For such monomial one has $w_{0,S}(f)=w_{n+1,S}(f)$ that will be simply denoted by $w_S(f)$. In particular, rephrasing  Proposition \ref{carat2}, $f\in I_{B_n}^{(m)}$ if and only if $w_S(f)\geq m$, for all $S\in \TT_{n-2,Q_n}$.
\end{Rmk}

\subsection{Generators of minimal degree in $I_{B_n}^{(m)}$}

For the rest of the paper we denote by $\alpha_{n,m}$ the minimal degree of a generator in the  $m-$th symbolic power of  $I_{B_n}$, that is $\alpha_{n,m}=\alpha(I_{B_n}^{(m)})$. By Proposition \ref{generators}, one has $\alpha_{n,1}=2$ for all $n$.

We first fix our attention in the case of a base with an odd number of vertices, i.e $n=2k-1$. We suppose $k\geq 3$ since for $k=2$, $I_{B_3}$ is a complete intersection ideal and $\gamma(I_{B_3})=2$.

\begin{Prop}\label{casodispari} One has $\alpha_{2k-1,s(2k-3)}=s(2k-1)$, for $s\in \NN$.
\end{Prop}

\begin{proof}
We prove the statement by induction on $s$.
As a step zero we consider $s=1$ and let $f=x_1\cdots x_{2k-1}$ be the product of the variables associated to all vertices of $Q_{2k-1}$. Clearly $w_S(f)=2k-3$ for all $S\in  \TT_{2k-3,Q_{2k-1}}$. Let us show that there are not monomial $g$ of degree $<2k-1$ in $I_{B_{2k-1}}^{(2k-3)}$. Consider a such monomial $g=x_0^{a_0}x_1^{a_1}\cdots x_{2k+2}^{a_{2k+2}}$ and suppose, by Lemma \ref{x0xn+1} that $a_0=a_{2k+2}=d> 0$. Then $\sum_{i_1}^{2k-1}a_i\leq 2k-2-2d$ and 
\[
w_S(g)=\sum_{i\in S}a_i+a_0 < 2k-2-2d+d=2k-2-d\leq 2k-3
\]
for at least one $S\in  \TT_{2k-3,Q_{2k-1}}$ hence $h\notin I_{B_{2k-1}}^{(2k-3)}$.
Thus, we can reduce to consider a monomial  $g$, in the variables $x_1,\dots, x_{2k-1}$,  with $\deg(h)=2k-2$.
We start with the case of $g$ consisting of $2k-2$ variables. Then we can choose a tree $S$ avoiding the vertex associated to one of these variables obtaining  $w_{S}(g)< 2k-4$ and so  $h\notin I_{2k-1}^{(2k-3)}$.
If $g$ has, at least, one variable $x_i$ of degree  $a_i>1$, then we choose the tree $S$ avoiding $\{i\}$ and we obtain $w_{S}(g)\leq 2k-2-a_1< 2k-3$ and hence $g\notin I_{2k-1}^{(2k-3)}$. 

Suppose now the statement is true for $s-1$ and we prove it for $s$.
Clearly $(x_1\cdots x_{2k-1})^s\in I_{B_{2k-1}}^{(s(2k-3))}$ and $\deg(f^s)=s(2k-1)$. Thus it is enough to prove that,  in $I_{B_{2k-1}}^{(s(2k-3))}$, there are not monomials of degree $s(2k-1)-1$. Suppose there exists such $h=x_0^{a_0}x_1^{a_1}\cdots x_{2k+2}^{a_{2k+2}}\in I_{B_{2k-1}}^{(s(2k-3))}$ with $\deg(h)=s(2k-1)-1$.  Suppose first $a_0=a_{2k}=d> 0$. Hence $\sum_{i=1}^{2k-1}a_i=s(2k-1)-1-2d<s(2k-1)-3$. We distinguish two cases.
\begin{itemize}
\item[(i)] If $a_1\not=0$ for all $i=1, \dots, n$ we can apply Lemma \ref{x1xn} obtaining a monomial $\tilde{g}$ of degree $s(2k-1)-1-(2k-1)=(s-1)(2k-1)-1$ with 
\[
w_S(\tilde{g})\geq s(2k-3)-(2k-3)=(s-1)(2k-3).
\]
Hence $\tilde{g}\in I_{B_{2k-1}}^{((s-1)(2k-3))}$ and this contradicts the inductive hypothesis since $\deg(\tilde{g})<(s-1)(2k-1)$.
\item[(ii)] Suppose now there exists at least one index $i\in \{1, \dots, 2k-1\}$ such that $a_i=0$. Let $\sigma\in S_{2k-1}$ be the permutation sending $j$ to $j+1$ and define $g_\sigma$ the monomial obtained by applying $\sigma$ to the variables in $h$. Then define the monomial
$h=\Pi_{t=1}^{2k-1}g_{\sigma^t}$. One has $\deg(h)=(2k-1)[s(2k-1)-1]$ and, by Remark \ref{rmksomma}, $w_S(h)\geq (2k-1)s(2k-3)$. Moreover, the exponent of each variable associated to a vertex of  $Q_{2k-1}$ has degree $s(2k-1)-2d-1$.  We apply $s(2k-1)-2d-1$ times Lemma \ref{x1xn}  to the monomial $h$ obtaining a monomial $\tilde{h}$ such that
\[
\tilde{h}=(x_0x_{2k})^{d(2k-1)}
\]
and
\[
w_S(\tilde{h})\geq (2k-1)s(2k-3)-(s(2k-1)-2d-1)(2k-3)=4dk-6d+2k-3 \quad \forall S\in \TT_{2k-3,Q_{2k-1}}
\]
which is a contradiction since, for $d>0$ and $k>3$ one has $w_S(\tilde{h})=d(2k-1)<4dk-6d+2k-3$.
\end{itemize}
Thus we can consider $a_0=a_{2k}=0$.  We distinguish two cases.
\begin{itemize}
\item[1)] $a_i\not=0$ for all $i=1,\dots, 2k-1$. In such case we can apply Lemma \ref{x1xn} to $g$ obtaining a new monomial  $\tilde{g}$ of degree $\leq (s-1)(2k-1)-1$ such that $w_S(\tilde{g})\geq s(2k-3)-k-3=(s-1)(2k-3)$. Thus, by Proposition \ref{impo}, $\tilde{g}\in  I_{B_{2k-1}}^{((s-1)(2k-3))}$ contradicting the inductive hypothesis.
\item[2)] There exists at least one index $i\in \{1, \dots, 2k-1\}$ such that $a_i=0$. Then define, as before, the monomial
$h=\Pi_{t=1}^{2k-1}g_{\sigma^t}$. One has $\deg(h)=(2k-1)[s(2k-1)-1]$ and $w_S(h)\geq (2k-1)s(2k-3)$. Moreover, the exponent of each variable associated to a vertex of  $Q_{2k-1}$ has degree $(s(2k-1)-1)$. Since for $s\geq 2$ one has $s(2k-1)-1>2ks-2s+1$ we apply $2sk-2s+1$ times Lemma \ref{x1xn}  to the monomial $h$ obtaining a monomial $\tilde{h}$ such that
\[
\deg(\tilde{h})=s(2k-1)-4k+1=(s-1)(2k-1)-2k+1
\]
and
\[
w_S(\tilde{h})\geq (2k-1)s(2k-3)-(2sk-2s+1)(2k-3)=(2k-3)(s-1) \quad \forall S\in \TT_{2k-3,Q_{2k-1}}
\]
But $\deg(\tilde{h})<(s-1)(2k-1)$ contradicts the inductive hypothesis.
\end{itemize}
\end{proof}

We focus now in the case of a base with an even number of vertices, i.e $n=2k$.

\begin{Prop}\label{casopari} One has $\alpha_{2k,s(k-1)}=sk$ for $s \in \NN$.
\end{Prop}

\begin{proof}
We prove the statement by induction on $s$. For the case $s=1$, since the bipyramid is build over a polygon with $2k$ vertices, we can consider the following monomials
\[
f_1=x_1x_3\cdots x_{2k-1}
\]
and
\[
f_2=x_2x_4\cdots x_{2k}.
\]
We easily verify that
\[
w_S(f_1)=w_S(f_2)=k-1, \quad \forall S\in \TT_{2k-2,Q_{2k}},
\]
hence $f,g \in I_{2k}^{(k-1)}$.

Let us show now that $I_{B_{2k}}^{(k-1)}$ does not contain any monomial of degree strictly less than $k$.
For this aim, let  $g=x_0^{a_0}x_1^{a_1}\cdots x_{2k+1}^{a_{2k+1}}\in I_{B_{2k}}^{(k-1)}$ with $\deg(g)=k-1$.
By Lemma \ref{x0xn+1} we first suppose that   $a_0=a_{2k+1}=d>0$, hence $\sum_{i=1}^{2k}a_i=k-1-2d$. Then
\[
w_S(g)=\sum_{i\in S}a_i+a_0\leq k-1-2d+d=k-1-d
\]
for at least one $S\in  \TT_{2k-2,Q_{2k}}$ hence $g\notin I_{B_{2k}}^{(k-1)}$.

Thus, we can reduce to consider a monomial  $g$, in the variables $x_1,\dots, x_{2k}$,  with $\deg(h)=k-1$.
We start with the case of $g$ consisting of $k-1$ variables. Then we can choose a tree $S$ avoiding the vertex associated to one of these variables obtaining  $w_{S}(g)< k-1$ and so  $h\notin I_{B_{2k}}^{(k-1)}$.
If $g$ has, at least, one variable $x_i$ of degree  $a_i$, then we choose the tree $S$ avoiding $\{i\}$ and we obtain $w_{S}(h)\leq k-1-d$ and hence $h\notin I_{B_{2k}}^{(k-1)}$. 

Suppose now the statement is true for $s-1$ and we prove it for $s$.
For the case $s\geq 2$  let $f_i$ be a monomial of degree  $k$  such that  $f \in I_{B_{2k}}^{(k-1)}$, by case $s=1$. Hence $f^s$ is in $I_{B_{2k}}^{(s(k-1))}$.  Thus it is enough to prove that,  in $I_{B_{2k}}^{(s(k-1))}$, there are not monomials of degree $sk-1$. Suppose there exists such $g=x_0^{a_0}x_1^{a_1}\cdots x_{2k+2}^{a_{2k+2}}\in I_{B_{2k}}^{(s(k-1))}$ with $\deg(h)=sk-1$.  Suppose first $a_0=a_{2k+1}=d> 0$. Hence $\sum_{i=1}^{2k}a_i=sk-1-2d$. We distinguish two cases.
\begin{itemize}
\item[(i)] If $a_1\not=0$ for all $i=1, \dots, n$ we can apply Lemma \ref{x1xn} obtaining a monomial $\tilde{g}$ of degree $sk-1-2k=(s-2)k-1$ with 
\[
w_S(\tilde{g})\geq s(k-1)-(2k-2)=(s-2)(k-1).
\]
Hence $\tilde{g}\in I_{B_{2k}}^{((s-2)(k-1))}$ and this contradicts the inductive hypothesis since $\deg(\tilde{g})<(s-2)(k-1)$.
\item[(ii)] Suppose now there exists at least one index $i\in \{1, \dots, 2k\}$ such that $a_i=0$. Let $\sigma\in S_{2k}$ be the permutation sending $j$ to $j+1$ and define $g_\sigma$ the monomial obtained by applying $\sigma$ to the variables in $h$. Then define the monomial
$h=\Pi_{t=1}^{2k}g_{\sigma^t}$. One has $\deg(h)=(2k)[sk-1]$ and, by Remark \ref{rmksomma}, $w_S(h)\geq 2ks(k-1)$. Moreover, the exponent of each variable associated to a vertex of  $Q_{2k}$ has degree $sk-1-2d$.  We apply $sk-1-2d$ times Lemma \ref{x1xn}  to the monomial $h$ obtaining a monomial $\tilde{h}$ such that
\[
\tilde{h}=(x_0x_{2k+1})^{2dk}
\]
and
\[
w_S(\tilde{h})\geq 2ks(k-1)-(sk-1-2d)(2k-2)=4dk-4d+2k-2 \quad \forall S\in \TT_{2k-2,Q_{2k}}
\]
which is a contradiction since, for $d>0$ and $k\geq 2$ one has $w_S(\tilde{h})=2dk<4dk-4d+2k-2$.
\end{itemize}
Thus we can consider $a_0=a_{2k+1}=0$.  We distinguish two cases.
\begin{itemize}
\item[1)] $a_i\not=0$ for all $i=1,\dots, 2k$. In such case we can apply Lemma \ref{x1xn} to $g$ obtaining a new monomial  $\tilde{g}$ of degree $\leq (s-2)k-1$ such that $w_S(\tilde{g})\geq s(k-1)-(2k-2_=(s-2)(2k-3)$. Thus, by Proposition \ref{impo}, $\tilde{g}\in  I_{B_{2k}}^{((s-2)(k-1))}$ contradicting the inductive hypothesis.
\item[2)] There exists at least one index $i\in \{1, \dots, 2k\}$ such that $a_i=0$.  Then define, as before, the monomial
$h=\Pi_{t=1}^{2k}g_{\sigma^t}$. One has $\deg(h)=(2k)[sk-1]$ and $w_S(h)\geq 2ks(k-1)$. Moreover, the exponent of each variable associated to a vertex of  $Q_{2k-1}$ has degree $(sk-1)$. We apply $sk-2$ times Lemma \ref{x1xn}  to the monomial $h$ obtaining a monomial $\tilde{h}$ such that
\[
\tilde{h}=x_1\cdots x_{2k}
\]
and
\[
w_S(\tilde{h})\geq 2ks(k-1)-(sk-2)(2k-2)=4(k-1) \quad \forall S\in \TT_{2k-2,Q_{2k}}
\]
which is a contradiction since $w_S(\tilde{h})=2(k-1)$.
\end{itemize}
\end{proof}

We  prove that, increasing the number of vertices in the base, the minimal degree of a generator in the $m-$th symbolic power does not increase.

\begin{Prop}\label{alfa1}
$\alpha_{n,m} \geq \alpha_{n+1,m}$
\end{Prop}

\begin{proof}
Suppose $\alpha_{n,m}=k$, then there exists a monomial  $f\in k[x_0, \dots, x_{n+1}]$ of degree $k$, such that $f\in I^{(m)}_n$. Moreover, by Proposition \ref{carat2} and Remark  \ref{impo} one has $w_{S}(f)\geq m$, $\forall S\in \TT_{n-2,Q_n}$. Let $\tilde{f} \in k[x_0, \dots, x_{n+2}]$ the monomial obtained by $f$ substituting the variable $x_{n+1}$ with $x_{n+2}$.
In the new $(n+1)-$gon $Q_{n+1}$, each subtree $S'$ of length $n+1-2=n-1$  containing  the new vertex $\{n+1\}$ (of $Q_{n+1}$ ) contains a tree  $S$ with $n-2$ vertices of the previous  $n-$gon $Q_n$. It is easy to observe that $w_{S'}(\tilde{f})=w_{S}(f) \geq m$. On the other hand, if $S'$ is a tree of length $n-1$ not containing $\{n+1\}$, then it contains two trees $S$ of  length $n-2$ for which $w_{S}(f) \geq m$, then we have $w_{S'}(\tilde{f})\geq w_{S}(f) \geq m$. In conclusion, we get $w_{S'}(\tilde{f})\geq m$, $\forall S'\in \TT_{n-1,Q_{n+1}}$. Hence  $\tilde{f}\in I_{n+1}^{(m)}$ from which we get $\alpha_{n+1,m}\leq k$.
\end{proof}

\begin{Rmk}\rm
When $s$ is even, there is another generator of minimal degree in $I_{B_{2k}}^{(s(k-1))}$ different from $f_1^s$ and $f_2^s$, where $f_1$ and $f_2$, according to Propositon \ref{casopari} are the two monomial defined by the product of variables associated to non-consecutive vertices in $Q_{2k}$. Let $s=2r$, then the other generator is $f_3=(x_1x_2\cdots x_{2k})^r$. As a matter of fact one has
\[
w_S(\tilde{h})=r(2k-2)=\frac{s}{2}(2k-2)=s(k-1) \quad \forall S\in \TT_{2k-2,Q_{2k}}.
\]
\end{Rmk}

Finally, using Propositions \ref{casopari} and \ref{alfa1} we can choose that, increasing the number of vertices in the base, the minimal degree of a generator in the $m-$th symbolic power becomes stationary. In \cite{franci} the reader can find many examples and conjectures about the 
behavior of $\alpha_{n,m}$ with respect to $\alpha_{n,m-1}$.
 
\begin{Prop} 
$\alpha_{2k+t,k-1}=k$ $\forall t\geq 0$.
\end{Prop}

\begin{proof}
From Propositions \ref{alfa1} and \ref{casopari} we know that $\alpha_{2k+t,k-1}\leq k$. We proof by induction on $t$ that we cannot have $\alpha_{2k+t,k-1}< k$. For $t=0$, it follows from Proposition \ref{casopari}. Suppose now that  $\alpha_{2k+t,k-1}=k$ and we want to prove that $\alpha_{2k+t+1,k-1}=k$. If  $\alpha_{2k+t+1,k-1}=k-1$ then there exists a monomial $f\in I_{B_{2k+t+1}}^{(k-1)}$ of degree $k-1$. By Proposition \ref{carat2} one has $w_S(f)\geq k-1$, for all $S\in \TT_{2k+t-1,Q_{2k+t+1}}$.  We can observe that the  $(2k+t-1)-$trees $S$ of $Q_{2k+t+1}$ containing the new vertex $\{2k+t+1\}$ are in correspondence 1:1 with the   $(2k+t-2)-$trees  $S'$ of $Q_{2k+t}$. Denote by $f'$ the image of  $f$ under the map which sends  the variable $x_{2k+t+1}$ to 1 and the  variable $x_{2k+t+2}$ (associated to the lower vertex of $B_{2k+t+1}$) to the variable $x_{2k+t+1}$ (associated to the lower vertex of $B_{2k+t}$). We distinguish two cases:
\begin{itemize}
\item[i)] $x_{2k+t+1}$ does not divide  $f$. In this situation one has $w_{S'}(f')=w_S(f)\geq k-1$ and $\deg(f')=k-1$. Hence  $\alpha_{2k+t,k-1}=k-1$ contradicting the inductive hypothesis.
\item[ii)] $x_{2k+t+1}$ divides $f$. Since the number $2k+t+1$ of vertices of $Q_{2k+t+1}$ is greater than the multiplicity $k-1$ of the symbolic power, it follows that there exists at least another monomial $g$ with $\deg(g)=k-1$, $w_S(g)\geq k-1$, for all $S\in \TT_{2k+t-1,Q_{2k+t+1}}$ and such that $x_{2k+t+1}$ does not divide $g$. Thus we return to case (i) obtaining again a contradiction.
\end{itemize}
\end{proof}

\subsection{Waldschmidt constant  of $I_{B_n}^{(m)}$}
We finally conclude computing the Waldschmidt constant of $I_{B_n}$, showing that this constant depends only on the number $n$ of edge in the base of $B_n$, independently if $n$ is even or odd, even though the two cases have different behavior in the $\alpha_{n,m}$'s.

\begin{proof}[Proof of Theorem \ref{main}]
To compute the Waldschmidt constant we reduce to consider a subsequence of all possible $m$ in the definition of the limit.

Let $n=2k$ be an even number. From Proposition \ref{casopari} one has $\alpha_{2k,s(k-1)}=sk$, that can be written as
$$\alpha_{n,s(\frac{n}{2}-1)}=s\frac{n}{2}$$ 
posing $k=\frac{n}{2}$. Then
$$\gamma(I_{B_n})=\lim_{s \to \infty}\frac{\alpha_{n,s(\frac{n}{2}-1)}}{s(\frac{n}{2}-1)} =\lim_{s \to \infty} \frac{s\frac{n}{2}}{s(\frac{n}{2}-1)}=\frac{\frac{n}{2}}{(\frac{n}{2}-1)}=\frac{n}{n-2}$$

Let $n=2k-1$ be an odd number. Taking  $s(2k-3)$, from Proposition \ref{casodispari}, one has $\alpha_{2k-1,s(2k-3)}=s(2k-1)$. Then
$$\gamma(I_{B_n})=\lim_{s \to \infty}\frac{\alpha_{2k-1,s(2k-3)}}{s(2k-3)}=\lim_{s \to \infty}\frac{s(2k-1)}{s(2k-3)}=\frac{2k-1}{2k-3}=\frac{n}{n-2}$$
\end{proof}

\end{document}